\documentclass[11pt]{article}

\usepackage[width=16cm,  top=2.5cm, footskip=2cm]{geometry}

\usepackage{amsmath}
\usepackage{amssymb}

\usepackage{authblk}

\usepackage{amsfonts}
\usepackage{latexsym}
\usepackage{indentfirst}
\usepackage{lineno}
\usepackage{bm}
\usepackage{url}
\usepackage{amsthm}

\usepackage{makeidx}         % allows index generation
\usepackage{graphicx}        % standard LaTeX graphics tool
\usepackage{multicol}        % used for the two-column index
\usepackage[bottom]{footmisc}% places footnotes at page bottom

\newtheorem{theorem}{Theorem}[section]
\newtheorem{corollary}{Corollary}[section]

\def\VIZ#1{(\ref{#1})}      % use for references to formulae
\def\FISHER#1{{\mathcal I}\big(#1\big)}
\def\FISHERR#1{{\mathcal I}_{\mathcal{H}}\big(#1\big)}

\def\RELENTR#1#2{\mathcal{H}\left({#1}\SEP{#2}\right)}
\def\RELENTI#1#2{\mathcal{H}\left({#1}\SEP{#2}\right)}
\def\RELENT#1#2{\mathcal{R}\left({#1}\SEP{#2}\right)}
\def\SEP{{\,|\,}}           % use for seperator in set defs, like x\in R \SEP

\def\GPATHS{Q_{0:T}}
\def\PATHS{Q_{0:T}^{\theta}}
\def\IPATHS#1{Q_{#1}^{\theta}}

\def\GPATHSAPP{\bar{Q}_{0:T}}
\def\PATHSAPP{Q_{0:T}^{\theta+\epsilon}}
\def\IPATHSAPP#1{Q_{#1}^{\theta+\epsilon}}

\def\EQUIL{\mu^\theta}

\def\INIT{\nu^\theta}
\def\INITAPP{\nu^{\theta+\epsilon}}

\def\EXPECT{\mathbb{E}}

\def\var{\mathrm{Var}}
\def\STATE{\mathbf{x}}
\begin{document}

\title{Pathwise Sensitivity Analysis in Transient Regimes}

\author[1]{Georgios Arampatzis\thanks{arampatzis@math.umass.edu}}
\author[1]{Markos A. Katsoulakis\thanks{markos@math.umass.edu}}
\author[1]{Yannis Pantazis\thanks{pantazis@math.umass.edu}}
\affil[1]{Department of Mathematics and Statistics University of Massachusetts, Amherst, MA, USA}

\date{}

\maketitle

\abstract{The instantaneous relative entropy (IRE) and the corresponding instantaneous Fisher information matrix (IFIM) for transient stochastic processes are presented in this paper.  These novel tools for sensitivity analysis of stochastic models serve as an extension of the well known relative entropy rate (RER) and the corresponding Fisher information matrix (FIM) that apply to stationary processes. Three cases are studied here, discrete-time Markov chains, continuous-time Markov chains and stochastic differential equations. A biological reaction network is presented as a demonstration numerical example. }

%-------------------------------------------------------------------------------------------------------------------------------------------------------------
%				INTRODUCTION
%-------------------------------------------------------------------------------------------------------------------------------------------------------------

\section{Introduction}

Sensitivity analysis, for a general mathematical model, is defined to be the quantification of  system response to parameter perturbations.
Questions on the robustness, (structural) identifiability, experimental design, uncertainty quantification, estimation and  control can be addressed through sensitivity analysis \cite{Saltelli:08}. Moreover,  it is a necessary analysis tool for the study of kinetic models such as chemical and biochemical reaction networks \cite{Saltelli:08,Distefano:13}. 
The mathematical models considered in this paper are models that describe phenomena exhibiting stochasticity: stochastic differential equations (SDE), discrete-time and continuous-time  Markov chains (DTMC and CTMC).  Some of the mathematical tools for the sensitivity analysis of such systems include log-likelihood methods and Girsanov transformations \cite{Glynn:90,Nakayama:94,Plyasunov:07}, polynomial chaos \cite{Kim:07}, finite difference methods  \cite{Rathinam:10, Anderson:12} and pathwise sensitivity methods \cite{Khammash:12}. Closely related to the log-likelihood methods are various  linear-response-based approaches, for instance in the context of chemical kinetics \cite{Meskine:09}, as well as in recent mathematical work for linear response in non-equilibrium systems \cite{Maes:09,Maes:10}.

In \cite{Pantazis:Kats:13}, the authors propose a new methodology for the pathwise sensitivity analysis of complex stochastic stationary dynamics based on the Relative Entropy (RE) between path distributions and Relative Entropy Rate (RER). This two quantities provide a measure of the sensitivity of the entire time-series distribution. The space of all such time-series is referred in probability theory as the ``path space''. RER measures the loss of information per unit time in path space after an arbitrary perturbation of parameter combinations. Moreover, RER and the corresponding Fisher Information Matrix (FIM) become computationally feasible in certain cases as they admit explicit formulas.
In fact, it is been showed in \cite{Pantazis:Kats:13} that the proposed pathwise sensitivity analysis has the following properties: (a) it is rigorously valid for the sensitivity of long-time, stationary dynamics, (b) it is a gradient-free sensitivity analysis method suitable for high-dimensional parameter spaces, (c) the computation of RER and FIM does not require the explicit knowledge of the equilibrium probability distribution, relying only on information for local dynamics, i.e. it is suitable for non-equilibrium systems.

In this paper we extend the RE and FIM tools, developed in \cite{Pantazis:Kats:13} for stationary processes, to transient and non-stationary processes. The extension is based on the notion of instantaneous RE (IRE) and instantaneous FIM (IFIM), see for example equations (\ref{inst:Rel:entr:DTMC}) and (\ref{inst:FIM:DTMC}) in text. These sensitivity tools provide an instantaneous measure for the sensitivity of a system arising from information theoretic tools.  Moreover, both IRE and IFIM, which are independent of observable functions, can be used as upper bounds for observable depended sensitivities, see for example the discussion in Conclusions and \cite{DKPP:2014,AKP:2015}.

The rest of the paper is organized as follows: In Section 2 a review of the RE and the FIM is given. In Sections 3,4 and 5 the new quantities IRE and IFIM are presented for discrete-time Markov chains, continuous-time Markov chains  and stochastic differential equations, respectively. In Section 6, a numerical example for a biological reaction network is utilized and IRE and IFIM are computed in the course of a stochastic simulation and various observations are discussed. Finally, in Section 7, concluding remarks and connections with existing works are discussed.

%-------------------------------------------------------------------------------------------------------------------------------------------------------------
%				END  INTRODUCTION
%-------------------------------------------------------------------------------------------------------------------------------------------------------------

\section{Time-dependent sensitivity analysis}

\subsection{Decomposition of the pathwise relative entropy}
The relative entropy (or Kullback-Leibler divergence) of a probability measure $P$
with respect to (w.r.t.) another probability measure $\bar{P}$ is defined as \cite{Kullback:59, Cover:91}
\begin{equation}
\RELENT{P}{\bar{P}} := \left\{ \begin{array}{ll}
\int \log{\frac{dP}{d\bar{P}}} d P \ , & \ \ P\ll \bar{P} \\
\infty \ , & \ \ \text{otherwise}
\end{array} \right.
\label{RE:def}
\end{equation}
where $\frac{dP}{d\bar{P}}$ is a function known as the Radon-Nikodym derivative which is well-defined
when $P$ is absolutely continuous w.r.t. $\bar{P}$ (denoted as $P\ll \bar{P}$) while the integration is performed
w.r.t. the probability measure $P$. Relative entropy has been utilized in a diverge range of scientific
fields from statistical mechanics \cite{Kipnis:99} to telecommunications \cite{Cover:91} and finance
\cite{Avellaneda:97} and possesses the following fundamental properties:
\begin{itemize}
\item[(i)] \ \ \ $\RELENT{P}{\bar{P}} \ge 0$,
\item[(ii)] \ \ \ $\RELENT{P}{\bar{P}} = 0$ if and only if $P= \bar{P} \;\;P$-almost everywhere, and,
\item[(iii)] \ \ \ $\RELENT{P}{\bar{P}}<\infty$ if and only if $P$ and $\bar{P}$ are absolutely continuous w.r.t. each other.
\end{itemize}
These properties allow us to view relative entropy  as a
``distance" (more precisely a divergence) between two probability measures capturing the relative
importance of uncertainties \cite{Liu:06}. From an information theory
perspective, relative entropy measures the {\it loss/change of information}
when $\bar{P}$ is considered instead of $P$ \cite{Cover:91}.

Let a stochastic process --either discrete-time or continuous-time-- be denoted by $X_t$ and let
the path space $\mathcal X$ be the set of all trajectories $\{X_t \}_{t=0}^T$. We denote by $\GPATHS$
the path space distribution, i.e., the probability to see a particular element of path space, $\mathcal X$.
Denote by $\GPATHSAPP$ the path space distribution of another process, $\bar{X}_t$. The
pathwise relative entropy of the distribution $\GPATHS$ w.r.t. the distribution $\GPATHSAPP$
assuming that they are absolutely continuous w.r.t. each other is written using \VIZ{RE:def} as
\begin{equation}
\RELENT{\GPATHS}{\GPATHSAPP} = \int \log \frac{d\GPATHS}{d\GPATHSAPP}  d\GPATHS \ .
\label{pathwise:RE:gen}
\end{equation}

A key property of pathwise relative entropy is that it is an increasing function of time which is the
analog of the second thermodynamic law in statistical physics \cite{Cover:91}. To this end, for
various important Markov processes, it can be shown, exploiting the Markov property, that
the pathwise relative entropy can be written as an averaged quantity as
\begin{equation}
\RELENT{\GPATHS}{\GPATHSAPP} = \left\{ \begin{array}{ll}
\RELENT{\nu}{\bar{\nu}} + \sum_{i=1}^T \RELENTR{Q_{i}}{\bar{Q}_{i}} \ , & \text{\ \ if time is discrete} \\
\RELENT{\nu}{\bar{\nu}} + \int_0^T  \RELENTR{Q_{t}}{\bar{Q}_{t}} dt \ , & \text{\ \ if time is continuous} \ ,
\end{array} \right.
\end{equation}
where $\RELENT{\nu}{\bar{\nu}}$ is the relative entropy of the initial distribution $\nu$
w.r.t. the perturbed initial distribution $\bar{\nu}$, while
$\RELENTI{Q_{\cdot}}{\bar{Q}_{\cdot}}$ denotes the instantaneous relative entropy
(for explicit formulas we refer to equation \VIZ{inst:Rel:entr:DTMC},
\VIZ{inst:Rel:entr:CTMC} and \VIZ{inst:Rel:entr:SDE}).
We named the quantity $\RELENTI{Q_{\cdot}}{\bar{Q}_{\cdot}}$
instantaneous relative entropy because it is also non-negative but more importantly is the
time-derivative of the pathwise relative entropy and inherits the properties of the relative entropy.
Notice that the notation for instantaneous relative entropy is somewhat confusing but it will
become clear when specific examples will be presented in the following sections.

Moreover the relative entropy rate is defined for a broad class of stochastic processes including
Markov and semi-Markov processes as \cite{Limnios:01} as the limit
\begin{equation}
\RELENTR{Q_{}}{\bar{Q}_{}} = \lim_{T\rightarrow\infty} \frac{1}{T}\RELENT{\GPATHS}{\GPATHSAPP} \ .
\end{equation}
Relative entropy rate is closely related to the instantaneous relative entropy in two
fundamental ways, Firstly it is the limit of instantaneous relative entropy as time goes to
infinity, whenever this limit is defined and second it holds at the stationary regime that
\begin{equation}
\RELENTI{Q_{i}}{\bar{Q}_{i}} = \RELENTR{Q_{}}{\bar{Q}_{}}
\end{equation}
for the discrete-time case with $i=1,...,T$ and similarly for the continuous-time case.

\subsection{Sensitivity analysis and Fisher information matrix}
In \cite{Pantazis:Kats:13}, authors proposed the relative entropy between path distributions to
perform sensitivity analysis arguing that pathwise relative entropy takes into account not only
the equilibrium properties but also the complex dynamics of a stochastic process. Specifically,
denote by $\PATHS$ the path space distribution of the process $X_t$, parametrized by the parameter vector
$\theta\in\mathbb R^K$. Consider also a perturbation vector, $\epsilon\in\mathbb R^K$, and denote
by $\PATHSAPP$ the path space distribution of the perturbed process $\bar{X}_t$. The
pathwise relative entropy of the distribution $\PATHS$ w.r.t. the distribution $\PATHSAPP$
assuming that they are absolutely continuous w.r.t. each other is written from \VIZ{pathwise:RE:gen} as
\begin{equation}
\RELENT{\PATHS}{\PATHSAPP} = \int \log \frac{d\PATHS}{d\PATHSAPP}  d\PATHS \ .
\label{pathwise:RE}
\end{equation}

An attractive approach to sensitivity analysis that is rigorously based on relative entropy
calculations is the Fisher Information Matrix (FIM). Indeed, assuming smoothness in the parameter
vector, it is straightforward to obtain the following expansion for \VIZ{pathwise:RE} \cite{Cover:91, Abramov:05},
\begin{equation}
\RELENT{{\PATHS}}{{\PATHSAPP}} = \frac{1}{2} \epsilon^T \FISHER{\PATHS} \epsilon + O(|\epsilon|^3)\, , 
\label{GFIM0}
\end{equation}
where the $K\times K$ pathwise FIM  $\FISHER{\PATHS}$ is defined
as the Hessian of the pathwise relative entropy.
As \VIZ{GFIM0} readily suggests, relative entropy is locally a quadratic function of the parameter
vector $\theta$. Indeed, the pathwise RE for any perturbation can be recovered up to third-order
utilizing only the pathwise FIM. Moreover, similar expansions hold for the instantaneous relative
entropy and the relative entropy rate. The decomposition of the pathwise FIM reads
\begin{equation}\label{pathwise:FIM}
\FISHER{\PATHS} = \left\{ \begin{array}{ll}
\FISHER{\INIT} + \sum_{i=1}^T \FISHERR{\IPATHS{i}} \ , & \text{\ \ if time is discrete} \\
\FISHER{\INIT} + \int_0^T  \FISHERR{\IPATHS{t}} dt \ , & \text{\ \ if time is continuous}
\end{array} \right.
\end{equation}
where $\FISHER{\INIT}$ is the FIM of the initial distribution $\INIT$ while $\FISHERR{\IPATHS{\cdot}}$
denotes the instantaneous FIM. In the following sections concrete examples of
stochastic Markov processes are presented whose pathwise relative entropy and the associated pathwise FIM
is provided.

\section{Discrete-time Markov chains}
\label{DTMC:app}
This section presents explicit formulas of the various relative entropy quantities defined in the previous
section as well as the associated FIMs for the case of discrete-time Markov chains.
The analysis of the DTMC case serves (a) as a more intuitive and manageable example of stochastic
processes and (b) as a intermediate step to handle the continuous-time Markov chain case.

Next, let $\{x_i\}_{i\in\mathbb Z^+}$ be a discrete-time time-homogeneous Markov chain with
separable state space $E$. The transition probability kernel of the Markov chain denoted
by  $P^\theta(x, dx')$ depends on the parameter vector $\theta\in\mathbb R^K$.
Assume that the transition kernel is absolutely continuous with respect to  the Lebesgue
measure and the transition probability density function $p^\theta(x,x')$
is always positive for all $x,x'\in E$ and for all $\theta\in\mathbb R^K$.
Exploiting the Markov
property, the path space probability density $\PATHS$ for the path $\{x_i\}_{i=0}^T$
at the time horizon $0,1,\ldots,T$ starting  from the initial distribution 
 $\INIT(x)dx$ is given by
\begin{equation*}
\PATHS\big(x_0,\ldots, x_T \big) = \INIT(x_0) p^\theta(x_0,x_1) \ldots  p^\theta(x_{T-1},x_T)\, .
\end{equation*}
We consider a perturbation vector $\epsilon\in\mathbb R^K$ and the Markov chain
$\{\bar{x}_i\}_{i\in\mathbb Z^+}$ with  transition probability
density function, $p^{\theta+\epsilon}(x,x')$,  initial density, $\INITAPP(x)$, as well as 
path distribution $\PATHSAPP$.
The product representation of the path distributions results in an additive representation
of the relative entropy of the path distribution $\PATHS$ w.r.t. the perturbed
path distribution $\PATHSAPP$. Let $\nu_{i}^\theta(x)$ denote the probability density function
of the Markov chain at time instant $i$ given that the initial distribution is $\INIT$, whose
formula is provided by the Chapman-Kolmogorov equation,
\begin{equation*}
\nu_{i}^\theta(x) = \int_E\cdots\int_E \INIT(x_0)p^\theta(x_0,x_1)\ldots p^\theta(x_{i-1},x) dx_0\ldots dx_{i-1} \ .
\end{equation*}
The following theorem presents the decomposition of the pathwise relative entropy.

\begin{theorem}
(a) The pathwise relative entropy for the above-defined discrete-time Markov chain
is decomposed as
\begin{equation}
\RELENT{\PATHS}{\PATHSAPP}= \RELENT{\INIT}{\INITAPP} + \sum_{i=1}^T \RELENTR{\IPATHS{i}}{\IPATHSAPP{i}} \ ,
\label{Rel:entr:DTMC}
\end{equation}
where the instantaneous relative entropy equals to
\begin{equation}
\RELENTR{\IPATHS{i}}{\IPATHSAPP{i}}
= \mathbb E_{\nu_{i-1}^\theta}\Big[\int_E p^\theta(x,x')\log \frac{p^\theta(x,x')}{p^{\theta+\epsilon}(x,x')}dx'\Big]
\label{inst:Rel:entr:DTMC}
\end{equation}
(b) Under smoothness assumption on the transition probability function
for the parameter $\theta$, the pathwise FIM is also decomposed as
\begin{equation}
\FISHER{\PATHS} = \FISHER{\INIT} + \sum_{i=1}^T  \FISHERR{\IPATHS{i}} \ ,
\label{path:FIM:DTMC}
\end{equation}
where the instantaneous pathwise FIM is given by
\begin{equation}
\FISHERR{\IPATHS{i}} =
\mathbb E_{\nu_{i-1}^\theta}[\int_E p^\theta(x,x')\nabla_\theta \log p^\theta(x,x')\nabla_\theta \log p^\theta(x,x')^T dx'] \ 
\label{inst:FIM:DTMC}
\end{equation}
\end{theorem}

\begin{proof}
(a) The proof of this part of the theorem can be found in \cite[Ch. 2]{Cover:91} under
the title ``Chain rule for relative entropy''
but for the shake of completeness we present it here. The Radon-Nikodym derivative
of the unperturbed path distribution w.r.t. the perturbed path distribution takes the form
\begin{equation*}
\frac{d\PATHS}{d\PATHSAPP}\big( \{x_i\}_{i=0}^T \big) =
\frac{\INIT(x_0)\prod_{i=0}^{T-1}p^{\theta}(x_i,x_{i+1})}{\INITAPP(x_0)\prod_{i=0}^{T-1}p^{\theta+\epsilon}(x_i,x_{i+1})}\, ,
\end{equation*}
which is well-defined since the transition probabilities are always positive. Then,
\begin{equation*}
\begin{aligned}
&\RELENT{\PATHS}{\PATHSAPP} \\
&= \int_E\cdots\int_E \INIT(x_0)\prod_{j=1}^T p^\theta(x_{j-1},x_j)
\log \frac{\nu^\theta(x_0)\prod_{i=1}^T p^\theta(x_{i-1},x_i)}
{\INITAPP(x_0)\prod_{i=1}^T p^{\theta+\epsilon}(x_{i-1},x_i)} dx_0\ldots dx_T \\
&= \int_E\cdots\int_E \INIT(x_0)\prod_{j=1}^T p^\theta(x_{j-1},x_j) \log \frac{\INIT(x_0)}{\INITAPP(x_0)} dx_0\ldots dx_T \\
& \quad + \sum_{i=1}^T \int_E\cdots\int_E \INIT(x_0)\prod_{j=1}^T p^\theta(x_{j-1},x_j) 
\log \frac{p^\theta(x_{i-1},X_i)}{p^{\theta+\epsilon}(x_{i-1},x_i)}  dx_0\ldots dx_T \\
%
%&= \int_E \INIT(x_0) \log \frac{\INIT(x_0)}{\INITAPP(x_0)} dx_0 \\
%& \quad + \sum_{i=1}^T \int_E\cdots\int_E \INIT(x_0)\prod_{j=1}^{i} p^\theta(x_{j-1},x_j)
%\log \frac{p^\theta(x_{i-1},x_i)}{p^{\theta+\epsilon}(x_{i-1},x_i)} dx_0 \ldots dx_i \\
%
&= \RELENT{\INIT}{\INITAPP} + \sum_{i=1}^T \RELENTR{\IPATHS{i}}{\IPATHSAPP{i}} \ ,
\end{aligned}
\end{equation*}
where $\RELENT{\INIT}{\INIT}=\mathbb E_{\INIT}\Big[\log \frac{\INIT(x)}{\INIT(x)}\Big]$
is the relative entropy of the unperturbed initial distribution w.r.t. the perturbed one,
while the instantaneous relative entropy (of the time-varying pathwise relative entropy) is
\begin{equation*}
\RELENTR{\IPATHS{i}}{\IPATHSAPP{i}} =
\mathbb E_{\nu_{i-1}^\theta}\Big[\int_E p^\theta(x,x')\log \frac{p^\theta(x,x')}{p^{\theta+\epsilon}(x,x')}dx'\Big] \ .
\end{equation*}

\noindent
(b) The proof of this part of the theorem is similar to the proof of the pathwise FIM
for the relative entropy rate in \cite{Pantazis:Kats:13}. We present it with minor but
necessary adaptations. Let $\delta p(x,x') = p^{\theta+\epsilon}(x,x') - p^\theta(x,x')$,
then the instantaneous relative entropy $\RELENTR{\IPATHS{i}}{\IPATHSAPP{i}}$
at the $i$-th time instant is written as
\begin{equation*}
\begin{aligned}
&\RELENTR{\IPATHS{i}}{\IPATHSAPP{i}}
= - \int_E\int_E \nu_{i-1}^\theta(x) p^\theta(x,x') \log \left(1+ \frac{\delta p(x,x')}{p^\theta(x,x')}\right) dx dx' \\
&= - \int_E\int_E \left[\nu_{i-1}^\theta(x)\delta p(x,x')
-  \frac{1}{2}\nu_i^\theta(x)\frac{\delta p(x,x')^2}{p^\theta(x,x')}
+ O(|\delta p(x,x')|^3) \right] dx dx' \, .
\end{aligned}
\end{equation*}
Moreover, for all $x\in E$, it holds that
\begin{equation*}
\int_E \delta p(x,x')dx' = \int_E p^{\theta+\epsilon}(x,x')dx'
- \int_E p^{\theta}(x,x')dx' = 1-1 = 0.
\end{equation*}
Since the transition probability function is smooth w.r.t. the parameter vector $\theta$,
a Taylor series expansion to $\delta p$ gives,
\begin{equation*}
\delta p(x,x') = \epsilon^T \nabla_\theta p^\theta(x,x') + O(|\epsilon|^2)\ .
\end{equation*}
Thus, we finally obtain for all $i=1,...,T$, that
\begin{equation*}
\begin{aligned}
&\RELENTR{\IPATHS{i}}{\IPATHSAPP{i}} 
= \frac{1}{2} \int_E\int_E  \nu_{i-1}^\theta(x) \frac{(\epsilon^T \nabla_\theta p^\theta(x,x'))^2}{p^\theta(x,x')} dx dx' + O(|\epsilon|^3) \\
&= \frac{1}{2}\epsilon^T\Big( \int_E\int_E \nu_{i-1}^\theta(x) p^\theta(x,x) \nabla_\theta \log p^\theta(x,x')
 \nabla_\theta \log p^\theta(x,x')^T dx dx' \Big) \epsilon + O(|\epsilon|^3) \\
&= \frac{1}{2} \epsilon^T \FISHERR{\IPATHS{i}} \epsilon + O(|\epsilon|^3)
\end{aligned}
\end{equation*}
where,
\begin{equation*}
\FISHERR{\IPATHS{i}} =  \mathbb E_{\nu_{i-1}^\theta} \left[\int_E p^\theta(x,x') \nabla_\theta \ ,
\log p^\theta(x,x') \nabla_\theta \log p^\theta(x,x')^T d\,x'\right]
\end{equation*}
is the instantaneous FIM associated to the instantaneous relative entropy.

Consequently, the pathwise FIM $\FISHER{\PATHS}$, i.e., the Hessian of the
pathwise relative entropy \VIZ{Rel:entr:DTMC} at point $\theta$, is given by
\begin{equation*}
\FISHER{\PATHS} = \FISHER{\INIT} + \sum_{i=1}^T  \FISHERR{\IPATHS{i}} \ ,
%\nu_i^{\theta}\otimes p^\theta
\end{equation*}
where $\FISHER{\INIT} = \mathbb E_{\INIT}[\nabla_\theta \log\INIT(x)\nabla_\theta \log\INIT(x)^T]$
is the FIM of the initial distribution.
\end{proof}

{\bf Remark 1:} Let `$\otimes$' denote the product operator of two distributions (i.e.,
$\nu\otimes p(A\times B) = \int_A p(x,B)\nu(x)dx$). Then, the instantaneous relative
entropy can be written as a relative entropy of the probability measure $\nu_{i-1}^{\theta}\otimes p^\theta$
w.r.t. the probability measure $\nu_{i-1}^{\theta}\otimes p^{\theta+\epsilon}$. Mathematically,
\begin{equation*}
\RELENTR{\IPATHS{i}}{\IPATHSAPP{i}} = \RELENT{\nu_{i-1}^{\theta}\otimes p^\theta}{\nu_{i-1}^{\theta}\otimes p^{\theta+\epsilon}} \ ,
\end{equation*}
and similarly for the associated instantaneous FIM it holds that
\begin{equation*}
\FISHERR{\IPATHS{i}} = \FISHER{\nu_{i-1}^{\theta}\otimes p^\theta} \ .
\end{equation*}

{\bf Remark 2:} The instantaneous relative entropy $\RELENTR{\IPATHS{i}}{\IPATHSAPP{i}}$
is different from and should not be confused with the relative entropy of the unperturbed distribution
at the $i$-th (or the $(i-1)$-th) time instant $\nu_i^{\theta}$ (or $\nu_{i-1}^{\theta}$) w.r.t. the respective
perturbed distribution $\nu_i^{\theta+\epsilon}$ (or $\nu_{i-1}^{\theta+\epsilon}$).
Indeed, it holds that
\begin{equation*}
\RELENTR{\IPATHS{i}}{\IPATHSAPP{i}} = \RELENT{\nu_{i-1}^{\theta}\otimes p^\theta}{\nu_{i-1}^{\theta}\otimes p^{\theta+\epsilon}}
\ne \RELENT{\nu_i^{\theta}}{\nu_i^{\theta+\epsilon}} \ ,
\end{equation*}
as well as $\RELENTR{\IPATHS{i}}{\IPATHSAPP{i}} \ne \RELENT{\nu_{i-1}^{\theta}}{\nu_{i-1}^{\theta+\epsilon}}$.
Moreover, an explicit formula for the probability distribution at the $i$-th time instant $\nu_i^{\theta}$
is generally not available making the computation of the relative entropy $\RELENT{\nu_i^{\theta}}{\nu_i^{\theta+\epsilon}}$
intractable. On the other hand, the instantaneous relative entropy $\RELENTR{\IPATHS{i}}{\IPATHSAPP{i}}$
can be computed in a straightforward manner as a statistical average since it incorporates only the transition
probabilities which are known functions.

\medskip
\noindent
{\bf Stationary regime}:
In the stationary regime, the initial distribution is the stationary distribution, $\EQUIL(\cdot)$. Thus,
for all $i=1,...,T$ it holds that $\nu_i^\theta=\EQUIL$ and the instantaneous relative entropy is a
constant function of time that  is equal to the relative entropy rate. The next corollary presents explicit
formulas for the relative entropy rate and the associated FIM.

\begin{corollary}
The relative entropy rate equals to
\begin{equation}
\RELENTR{\IPATHS{}}{\IPATHSAPP{}}
= \mathbb E_{\EQUIL}\Big[\int_E p^\theta(x,x')\log \frac{p^\theta(x,x')}{p^{\theta+\epsilon}(x,x')}dx'\Big] \ .
\end{equation}
Similarly, the FIM associated to the RER is given by
\begin{equation}
\FISHERR{\IPATHS{}} = \mathbb E_{\EQUIL} \left[\int_E p^\theta(x,x) \nabla_\theta
\log p^\theta(x,x') \nabla_\theta \log p^\theta(x,x')^T d\,x'\right] \ .
\end{equation}
\end{corollary}

\begin{proof}
Both formulas are obtained by substituting the stationary distribution $\EQUIL$ to the place of $\nu_{i-1}^\theta$
in \VIZ{inst:Rel:entr:DTMC} and \VIZ{inst:FIM:DTMC}.
\end{proof}

\section{Continuous-time Markov chains}
\label{trans:RER:app}
Let $\{X_t\}_{t\in\mathbb R_+}$ be a continuous-time Markov chain  
with countable state space $E$. The parameter dependent transition
rates, denoted by $c^\theta(x, x')$, completely define the continuous-time Markov chain.
The transition rates determine the updates (jumps or
sojourn times) from a current state $x$ to a new (random) state $x'$ through
the total rate $\lambda^\theta(x)=\sum_{x'\in E} c^\theta(x, x')$ which
is the intensity of the exponential waiting time for a jump from state $x$. The
transition probabilities for the embedded Markov chain $\big\{x_n\big\}_{n\geq0}$
defined by $x_n:=X_{t_n}$ where $t_n$ is the instance of the $n$-th jump are
$p^\theta(x, x') = \frac{c^\theta(x, x')}{\lambda^\theta(x)}$.
% The discrete-time
%Markov chain $\{\xi_n\}_{n\in\mathbb Z^+} defined for all $n$ by $\xi_n = (x_n,\tau_n)$
%is probabilistically equivalent to the continuous-time Markov chain ${X_t\}_{t\in\mathbb R_+}$.

Assume another jump Markov process $\{\bar{X}_t\}_{t\in\mathbb R_+}$, 
defined by perturbing the transition rates by a small vector $\epsilon\in\mathbb R^k$.
Moreover assume that the two path probabilities  $\PATHS$ and $\PATHSAPP$ are absolutely continuous
with respect to  each other which is satisfied when $c^{\theta}(x, x')=0$ if and
only if $c^{\theta+\epsilon}(x, x')=0,$  $\forall x,x'\in E$. The following theorem
presents the decomposition of the pathwise relative entropy for the case of continuous-time
Markov chains.

\begin{theorem}
(a) The pathwise relative entropy for the above-defined continuous-time Markov chain
is decomposed as
\begin{equation}
\RELENT{\PATHS}{\PATHSAPP} = \RELENT{\INIT}{\INITAPP}
+ \int_0^T \RELENTR{\IPATHS{t}}{\IPATHSAPP{t}} dt \ ,
\label{Rel:entr:CTMC}
\end{equation}
where the instantaneous relative entropy equals to
\begin{equation}
\RELENTR{\IPATHS{t}}{\IPATHSAPP{t}} = \mathbb E_{Q_{0:t}^{\theta}}\left[ \lambda^{\theta}(X_{t-})
\log \frac{c^\theta(X_{t-}, X_{t})}{c^{\theta+\epsilon}(X_{t-}, X_{t})} - \big(\lambda^\theta(X_{t}) - \lambda^{\theta+\epsilon}(X_{t})\big)\right] \ .
\label{inst:Rel:entr:CTMC}
\end{equation}
(b) Under smoothness assumption on the transition rate function $c^\theta(\cdot, \cdot)$
for the parameter $\theta$, the pathwise FIM is also decomposed as
\begin{equation}
\FISHER{\PATHS} = \FISHER{\INIT} + \int_0^T  \FISHERR{\IPATHS{t}} dt \ ,
\label{path:FIM:CTMC}
\end{equation}
where the instantaneous pathwise FIM is given by
\begin{equation}
\FISHERR{\IPATHS{t}} = \mathbb E_{Q_{0:t}^{\theta}}\left[ \lambda^{\theta}(X_{t-})
\nabla_\theta \log c^\theta(X_{t-}, X_{t}) \nabla_\theta \log c^{\theta}(X_{t-}, X_{t})^T \right] \ .
\label{inst:FIM:CTMC}
\end{equation}
\end{theorem}

\begin{proof}
(a) As in the discrete-time case, the key element is an explicit formula for the Radon-Nikodym derivative.
The Radon-Nikodym derivative of the path distribution $\PATHS$ w.r.t.  the perturbed
path distribution $\PATHSAPP$ has an explicit formula known also as Girsanov formula
\cite{Liptser:77, Kipnis:99},
\begin{equation*}\label{RN:MP}
\begin{aligned}
&\frac{d\PATHS}{d\PATHSAPP}  (\{X_t\}_{t=0}^T) =  \frac{\INIT(X_0)}{\INITAPP(X_0)}
\exp \left\{ \int_0^T \log \frac{c^\theta(X_{t-},X_t)}{c^{\theta+\epsilon}(X_{t-},X_t)}dN_t
- \int_{0}^{T} [\lambda^\theta(X_t) - \lambda^{\theta+\epsilon}(X_t)]\,dt  \right\} \ ,
\end{aligned}
\end{equation*}
where $\INIT$ (reps. $\INITAPP$) is the initial distributions  of $\{X_t\}_{t\in\mathbb R_+}$
(resp. $\{\bar{X}_t\}_{t\in\mathbb R_+}$) while $N_t$ is the counting  measure, i.e. counts the number of  jumps in the process up to time $t$.
Using the Girsanov formula, the pathwise relative entropy is rewritten as {\small
\begin{equation*}
\begin{aligned}
&\RELENT{\PATHS}{\PATHSAPP} \\
&= \mathbb E_{\PATHS} \left[ \log \frac{\INIT(X_0)}{\INITAPP(X_0)}
+\int_0^T \log \frac{c^\theta(X_{t-},X_t)}{c^{\theta+\epsilon}(X_{t-},X_t)}\,dN_t
- \int_{0}^{T} [\lambda^\theta(X_t) - \lambda^{\theta+\epsilon}(X_t)]\,dt \right] \\
&= \mathbb E_{\PATHS} \left[ \log \frac{\INIT(X_0)}{\INIT(X_0)} \right]
+ \mathbb E_{\PATHS} \left[ \int_0^T \log \frac{c^\theta(X_{t-},X_t)}{c^{\theta+\epsilon}(X_{t-},X_t)}\,dN_t \right] - \ldots \\
&\hspace{170pt} \ldots - \mathbb E_{\PATHS} \left[ \int_{0}^{T} [\lambda^\theta(X_t) - \lambda^{\theta+\epsilon}(X_t)]\,dt \right] \ .
\end{aligned}
\end{equation*}}
\noindent
Exploiting the fact that the process $M_T:=N_T-\int_0^t \lambda^{\theta}(X_{t})dt$
is a martingale, we have that {\small
\begin{equation*}
\mathbb E_{\PATHS} \left[ \int_0^T \log \frac{c^\theta(X_{t-},X_t)}{ac{\theta+\epsilon}(X_{t-},X_t)}\,dN_t \right]
= \mathbb E_{\PATHS} \left[ \int_0^T \lambda^{\theta}(X_{t-})
\log \frac{c^\theta(X_{t-},X_t)}{c^{\theta+\epsilon}(X_{t-},X_t)}\,dt \right] \ .
\end{equation*}}
Thus, the pathwise relative entropy is rewritten as {\small
\begin{equation*}
\begin{aligned}
&\RELENT{\PATHS}{\PATHSAPP} \\
&= \RELENT{\INIT}{\INITAPP} + \mathbb E_{\PATHS}\left[ \int_{0}^{T} \lambda^{\theta}(X_{t-})
\log \frac{c^\theta(X_{t-}, X_{t})}{c^{\theta+\epsilon}(X_{t-}, X_{t})}
- \big(\lambda^\theta(X_{t}) - \lambda^{\theta+\epsilon}(X_{t})\big)\,dt \right] \\
&= \RELENT{\INIT}{\INITAPP} + \int_{0}^{T} \mathbb E_{\PATHS}\left[ \lambda^{\theta}(X_{t-})
\log \frac{c^\theta(X_{t-}, X_{t})}{c^{\theta+\epsilon}(X_{t-}, X_{t})}
- \big(\lambda^\theta(X_{t}) - \lambda^{\theta+\epsilon}(X_{t})\big)\right]\,dt \\
%&= \RELENT{\INIT}{\INITAPP}
%+ \int_{0}^{T} \mathbb E_{Q_{0:t}^{\theta}}\left[ \lambda^{\theta}(X_{t-})
%\log \frac{c^\theta(X_{t-}, X_{t})}{c^{\theta+\epsilon}(X_{t-}, X_{t})}
%- \big(\lambda^\theta(X_{t}) - \lambda^{\theta+\epsilon}(X_{t})\big)\right]\,dt \\
&= \RELENT{\INIT}{\INITAPP} + \int_{0}^{T} \RELENTR{\IPATHS{t}}{\IPATHSAPP{t}} \,dt \ ,
\end{aligned}
\end{equation*}}
%\noindent
%where we use in the last equation the fact that
%\begin{equation*}
%\begin{aligned}
%\mathbb E_{\PATHS} \left[f(X_{t})\right] &=
%\sum_{n\ge0}\sum_{x_1\ldots x_{n+1}} \int_{\mathbb R_+^{n+1}} f(x_k) \chi\big\{ T_k\le t < T_{k+1} \big\}
%\prod_{i=0}^n P(\xi_i,\xi_{i+1}) d\tau_0 \ldots d\tau_{n+1} \\
%&= \sum_{k\ge0}\sum_{x_1\ldots x_{k+1}} \int_{\mathbb R_+^{n+1}} f(x_k) \chi\big\{ T_k\le t < T_{k+1} \big\}
%\prod_{i=0}^k P(\xi_i,\xi_{i+1}) d\tau_0\ldots d\tau_{k+1} \\
%&= \mathbb E_{Q_{0:t}^{\theta}} \left[f(X_{t})\right]
%\end{aligned}
%\end{equation*}
%
where the instantaneous relative entropy is defined as
\begin{equation*}
\RELENTR{\IPATHS{t}}{\IPATHSAPP{t}} = \mathbb E_{Q_{0:t}^{\theta}}\left[ \lambda^{\theta}(X_{t-})
\log \frac{c^\theta(X_{t-}, X_{t})}{c^{\theta+\epsilon}(X_{t-}, X_{t})}
- \big(\lambda^\theta(X_{t}) - \lambda^{\theta+\epsilon}(X_{t})\big)\right] \ .
\end{equation*}

\noindent
(b) Even though not directly evident from \VIZ{inst:Rel:entr:CTMC}, the instantaneous relative entropy
for continuous-time Markov chains is locally a quadratic function of the parameter vector
$\theta$. Indeed, defining the rate difference $\delta c(x,x') = c^{\theta+\epsilon}(x,x')
- c^\theta(x,x')$, the instantaneous relative entropy can be rewritten as
\begin{equation*}
\begin{aligned}
&\RELENTR{\IPATHS{t}}{\IPATHSAPP{t}} \\
&= - \EXPECT_{Q_{0:t}^{\theta}}\left[ \lambda^{\theta}(X_{t-})
\log \left( 1 + \frac{\delta c(X_{t-}, X_{t})}{c^{\theta}(X_{t-}, X_{t})}\right) \right]
- \EXPECT_{Q_{0:t}^{\theta}}\left[ \lambda^\theta(X_{t}) - \lambda^{\theta+\epsilon}(X_{t})\right] \\
&= - \EXPECT_{Q_{0:t}^{\theta}}\left[ \lambda^{\theta}(X_{t-})
\left( \frac{\delta c(X_{t-}, X_{t})}{c^{\theta}(X_{t-}, X_{t})} - \frac{1}{2} \left(\frac{\delta c(X_{t-}, X_{t})}{c^{\theta}(X_{t-}, X_{t})}\right)^2
+ O(|\delta c(X_{t-},X_t)|^3) \right)\right] \\
&+ \EXPECT_{Q_{0:t}^{\theta}}\left[ \lambda^{\theta+\epsilon}(X_{t}) - \lambda^\theta(X_{t})\big)\right] \\
%&= .....  \\
&= \frac{1}{2} \EXPECT_{Q_{0:t}^{\theta}}\left[ \lambda^{\theta}(X_{t-}) \right(\frac{\delta c(X_{t-}, X_{t})}{c^{\theta}(X_{t-}, X_{t})}\left)^2 \right]
+ O(|\delta c|^3) \\
\end{aligned}
\end{equation*}
Under smoothness assumption on the transition rates in a neighborhood of parameter
vector $\theta$ a Taylor series expansion of $\delta c(x,x') = \epsilon^T \nabla_\theta c^\theta(x,x') + O(|\epsilon|^2)$
results in
\begin{equation*}
\begin{aligned}
&\RELENTR{\IPATHS{t}}{\IPATHSAPP{t}} \\
&= \frac{1}{2} \EXPECT_{Q_{0:t}^{\theta}}\left[ \lambda^{\theta}(X_{t-}) \left(\frac{\delta c(X_{t-}, X_{t})}{c^{\theta}(X_{t-}, X_{t})}\right)^2 \right]
+ O(|\delta c|^3) \\
&= \frac{1}{2} \epsilon^T \EXPECT_{Q_{0:t}^{\theta}}\left[ \lambda^{\theta}(X_{t-})
\frac{\nabla_\theta c^\theta(X_{t-}, X_{t})\nabla_\theta c^\theta(X_{t-}, X_{t})^T}{c^{\theta}(X_{t-}, X_{t})^2} \right] \epsilon
+ O(|\epsilon|^3) \\
&= \frac{1}{2} \epsilon^T \FISHERR{\IPATHS{t}} \epsilon + O(|\epsilon|^3)
\end{aligned}
\end{equation*}
where 
\begin{equation*}
\FISHERR{\IPATHS{t}} = 
\mathbb E_{Q_{0:t}^{\theta}}\left[ \lambda^{\theta}(X_{t-})
\nabla_\theta \log c^\theta(X_{t-}, X_{t}) \nabla_\theta \log c^{\theta}(X_{t-}, X_{t})^T \right] \ ,
\end{equation*}
is the instantaneous FIM. Finally, the pathwise FIM is obtained from a straightforward expansion
of each element of the pathwise relative entropy in terms of $\epsilon$. It is given by
\begin{equation*}
\FISHER{\PATHS} = \FISHER{\INIT}  + \int_{0}^{T} \FISHERR{\IPATHS{t}} dt\ ,
\end{equation*}
where $\FISHER{\INIT}$ is the FIM of the initial distribution while
$\FISHERR{\IPATHS{t}}$ is the instantaneous pathwise FIM computed above.
\end{proof}

\medskip
\noindent
{\bf Stationary regime}: In the stationary regime, the instantaneous relative entropy is a constant function of time
since at each time instant the distribution of the states is the --typically unknown-- stationary distribution.
The following corollary presents explicit formulas for the relative entropy rate and the associated FIM at
the stationary regime.

\begin{corollary}
The relative entropy rate  is equal to the ergodic average
\begin{equation}
\RELENTR{\IPATHS{}}{\IPATHSAPP{}} = \EXPECT_{\EQUIL} \Big[\sum_{x'\in E} c^{\theta}(x, x')
\log \frac{c^\theta(x, x')}{c^{\theta+\epsilon}(x, x')} - (\lambda^{\theta}(x) - \lambda^{\theta+\epsilon}(x)) \Big] \ , 
\label{RER:CTMC}
\end{equation}
while the FIM of the relative entropy rate, computed as its Hessian, has explicit formula given by
\begin{equation}
\FISHERR{\IPATHS{}} = \EXPECT_{\EQUIL}\left[ \sum_{x'\in E} c^\theta(x,x')
\nabla_\theta \log c^\theta(x,x') \nabla_\theta \log c^\theta(x,x')^T \right] \ .
\label{RER:FIM:CTMC}
\end{equation}
\end{corollary}

\begin{proof}
At the stationary regime, the instantaneous relative entropy is rewritten as
\begin{equation}
\begin{aligned}
&\RELENTR{\IPATHS{t}}{\IPATHSAPP{t}} \\
&= \mathbb E_{Q_{0:t}^{\theta}}\left[ \lambda^{\theta}(X_{t-})
\log \frac{c^\theta(X_{t-}, X_{t})}{c^{\theta+\epsilon}(X_{t-}, X_{t})}
- \big(\lambda^\theta(X_{t}) - \lambda^{\theta+\epsilon}(X_{t})\big)\right] \\
%&= ... \\
&= \EXPECT_{\EQUIL} \Big[\sum_{x'\in E} c^{\theta}(x, x')
\log \frac{c^\theta(x, x')}{c^{\theta+\epsilon}(x, x')} - (\lambda^{\theta}(x) - \lambda^{\theta+\epsilon}(x)) \Big]\, .
\end{aligned}
\end{equation}

\end{proof}

\section{Stochastic differential equations - Markov processes}
Consider a Markov process $X_t\in\mathbb R^d$ driven by a stochastic differential equation of the form
\begin{equation}
 \left\{
 \begin{array}{l}
	dX_t = b^\theta(X_t) dt + \sigma(X_t) d{W}_t \\
	X_0 \sim \INIT
\end{array}  \right.
\label{sde:def}
\end{equation}
where $b^\theta(\cdot)$ is the drift function depending on the parameter vector $\theta$, $\sigma(\cdot)$
is the state-dependent diffusion matrix, ${W}_t$ is a $d$-dimensional Brownian motion while $\INIT$ is the
initial distribution of the process. Let $\PATHS$ denote the path space distribution for a specific parameter
vector $\theta$. Consider also a perturbation vector, $\epsilon \in \mathbb{R}^{K}$, and denote by $\PATHSAPP$
the path space distribution of the perturbed process, $\bar{X}_t$ driven by the stochastic differential
equation \VIZ{sde:def} with perturbed drift function.

Under appropriate assumption on the components of the stochastic differential equation,
the pathwise relative entropy can be decomposed as in the previous cases and an explicit
formula for the instantaneous relative entropy can be estimated as the following theorem
asserts.

\begin{theorem}
(a) Assume that the diffusion matrix, $\sigma(x)$, is invertible for all $x\in\mathbb R^d$ and
\begin{equation*}
\mathbb E_{\PATHS} [\exp\big\{\int_0^T |\sigma^{-1}(X_t)(b^{\theta+\epsilon}(X_t) - b^\theta(X_t))|^2  \big\}]<\infty \quad  \textrm{(Novikov condition)} \ .
\end{equation*}
Then, the pathwise relative entropy for the above-defined Markov process
is decomposed as
\begin{equation}
\RELENT{\PATHS}{\PATHSAPP} = \RELENT{\INIT}{\INITAPP} +  \int_0^T \RELENTR{\IPATHS{t}}{\IPATHSAPP{t}} dt \ ,
\label{Rel:entr:SDE}
\end{equation}
where the instantaneous relative entropy is equal to
\begin{equation}
\RELENTR{\IPATHS{t}}{\IPATHSAPP{t}} = \frac{1}{2} \mathbb E_{\nu_t^\theta} \left[ \big|\sigma^{-1}(x)\big(b^{\theta+\epsilon}(x) - b^\theta(x)\big)\big|^2 \right] \ .
\label{inst:Rel:entr:SDE}
\end{equation}
(b) Assume further that the drift function $b^\theta(\cdot)$ is smooth w.r.t. the parameter vector $\theta$.
Then, the pathwise FIM has a similar decomposition and the instantaneous FIM is given by
\begin{equation}
\FISHERR{\IPATHS{t}} = \mathbb E_{\nu_t^{\theta}}\left[
\big[\nabla_{\theta}b^{\theta}(x)^T(\sigma(x)\sigma(x)^T)^{-1}\nabla_{\theta}b^{\theta}(x) \right] \ .
\label{inst:FIM:SDE}
\end{equation}
where $\nabla_{\theta}b^{\theta}(\cdot)$ is a $d\times K$ matrix containing all the first-order partial
derivatives of the drift vector (i.e., the Jacobian matrix).
\label{theorem:SDE}
\end{theorem}

\begin{proof}
(a) The inversion of the diffusion matrix is a necessary assumption for the well-poshness of the
Novikov condition which in turn suffices for the two path distributions $\PATHS$ and $\PATHSAPP$
to be absolutely continuous w.r.t. each other \cite{Oksendal:00}. Additionally, the Girsanov theorem
provides an explicit formula of the Radon-Nikodym derivative \cite{Oksendal:00} which is given by
\begin{equation*}
\frac{d\PATHS}{d\PATHSAPP} \Big(\big\{(X_t)\big\}_{t=0}^T \Big)
= \frac{d\INIT}{d\INITAPP}(X_0) \exp\left\{-\int_0^T u(X_t)^TdW_t - \frac{1}{2}\int_0^T |u(X_t)|^2 dt \right\} \ ,
\end{equation*}
where $u(x) = \sigma^{-1}(x)\big(b^{\theta+\epsilon}(x) - b^\theta(x)\big)$. 
Furthermore, it holds that 
\begin{equation*}
\hat{W}_t:=\int_0^t u(X_s)dt + W_t \ ,
\end{equation*}
is a Brownian motion w.r.t. the unperturbed path distribution $\PATHS$, meaning that, for any measurable function $f(\cdot)$,
it holds $\mathbb E_{\PATHS} \big[\int_0^T f(X_t)^Td\hat{W}_t\big] = 0$. Then,
\begin{equation*}
\begin{aligned}
&\RELENT{\PATHS}{\PATHSAPP} =
 \EXPECT_{\PATHS} \left[ \log \frac{d\INIT}{d\INIT}(X_0) - \int_0^T u(X_t)^TdW_t - \frac{1}{2}\int_0^T |u(X_t)|^2 dt \right] \\
&= \EXPECT_{\PATHS} \left[ \log \frac{d\INIT}{d\INIT}(X_0) \right]
- \EXPECT_{\PATHS} \left[ \int_0^T u(X_t)^Td\hat{W}_t \right]
+ \frac{1}{2} \EXPECT_{\PATHS} \left[ \int_0^T |u(X_t)|^2 dt \right] \\
&= \RELENT{\INIT}{\INITAPP} + \frac{1}{2} \EXPECT_{\PATHS} \left[ \int_0^T |u(X_t)|^2 dt \right] \\
&= \RELENT{\INIT}{\INITAPP} + \int_0^T \frac{1}{2} \EXPECT_{\PATHS} \left[  |u(X_t)|^2 \right] dt \\
&= \RELENT{\INIT}{\INITAPP} + \int_0^T \frac{1}{2} \mathbb E_{\nu_t^\theta} |u(X_t)|^2 dt
\end{aligned}
\end{equation*}
hence the instantaneous relative entropy is explicitly given by
\begin{equation*}
\RELENTR{\IPATHS{t}}{\IPATHSAPP{t}}
= \frac{1}{2} \mathbb E_{\nu_t^\theta} \big[ \big|\sigma^{-1}(x)\big(b^{\theta+\epsilon}(x) - b^\theta(x)\big)\big|^2 \big] 
\end{equation*}

\noindent
(b) Due to the smoothness assumption, a Taylor expansion of the drift function around the point
$\theta$ results in $b^{\theta+\epsilon}(x) - b^\theta(x) = \nabla_{\theta}b^{\theta}(x)\epsilon + O(|\epsilon|^2)$
where $\nabla_{\theta}b^{\theta}(\cdot)$ is a $d\times K$ matrix containing all the first-order partial
derivatives of the drift vector function (i.e., the Jacobian matrix).
Then, it is straightforward to obtain from \VIZ{inst:Rel:entr:SDE} to get
\begin{equation*}
\FISHERR{\IPATHS{t}} = \mathbb{E}_{\nu_t^\theta}
\big[\nabla_{\theta}b^{\theta}(x)^T(\sigma(x)\sigma(x)^T)^{-1}\nabla_{\theta}b^{\theta}(x)\big] \ .
\end{equation*}
\end{proof}

\medskip
\noindent
{\bf Stationary regime}:
In the stationary regime, the instantaneous relative entropy becomes a constant function of time
since the distribution of the process equals to the stationary distribution denoted by $\EQUIL$ for
all times. The following corollary presents explicit formulas for the relative entropy rate and the
associated FIM.

\begin{corollary}\label{propos:RER}
Under the same assumption of Theorem~\ref{theorem:SDE} and let $X_0\sim\EQUIL$.
Then, the relative entropy rate equals to
\begin{equation}
\RELENT{\IPATHS{}}{\IPATHSAPP{}} =  \frac{1}{2} \mathbb{E}_{\EQUIL}
\left[ \big|\sigma^{-1}(x)\big(b^{\theta+\epsilon}(x) - b^\theta(x)\big)\big|^2 \right] \ ,
\label{equil:RER:SDE}	
\end{equation}
while the FIM associated with the relative entropy rate is given by
\begin{equation}
\FISHERR{\IPATHS{}} = \mathbb{E}_{\EQUIL}[\nabla_{\theta}b^{\theta}(x)^T(\sigma(x)\sigma(x)^T)^{-1}\nabla_{\theta}b^{\theta}(x)] \ .
\label{equil:FIM:SDE}	
\end{equation}
\end{corollary}

We finally remark that a popular method for modeling non-equilibrium systems in atomistic and
mesoscopic scales is based on the Langevin equation. Langevin equation is a degenerate system
of stochastic differential equations whose sensitivity analysis based on the relative entropy rate
and the associated pathwise FIM was performed in \cite{TPKH:2015}.

\section{Demonstration example} \label{demonstration}

In this section we give a numerical example of the pathwise relative entropy (\ref{pathwise:RE}) and pathwise FIM (\ref{pathwise:FIM}) for a continuous time Markov chain model. More specific, a biological reaction network is considered and the quantities instantaneous RE (\ref{inst:Rel:entr:CTMC}) and instantaneous FIM (\ref{inst:FIM:CTMC}) are presented as a function of time.

\subsection{Continuous time Markov chains: an EGFR model}\label{sec:egfr}

In \cite{Kholodenko:99},  Kholodenko et al.\ proposed a reaction network that describes signaling phenomena of mammalian cells \cite{Moghal:99, Hackel:99,Schoeberl:02}.
The reaction network consists of $N=23$ species and $M=47$ reactions. The propensity function for the $R_j$ reaction, $j=1,\ldots,47 \textrm{ and } j\neq 7,14,29$, obeys the law of mass action \cite{Distefano:13},
\begin{equation}\label{mass:action}
a_j(\STATE) = k_j \binom{\STATE_{A_j}}{\alpha_j} \binom{\STATE_{B_j}}{\beta_j},
\end{equation}
for a reaction of the general form ``$\alpha_j A_j + \beta_j B_j \xrightarrow{k_j} \ldots$'', where
$A_j$ and $B_j$ are the reactant species, $\alpha_j$ and $\beta_j$ are the respective number
of molecules needed for the reaction, $k_j$ the reaction constant and $\STATE_{A_j}$ and $\STATE_{B_j}$ is
the total number of species $A_j$ and $B_j$, respectively. The binomial coefficient
is defined by $\binom{n}{k}=\frac{n!}{k!(n-k)!}$. The propensity functions for reactions  $R_{7},R_{14},R_{29}$ are being described by the Michaelis--Menten kinetics, see \cite{Distefano:13},
\begin{equation}
a_j(\STATE) = V_j \STATE_{A_j} / \left(  K_j + \STATE_{A_j} \right), \quad j=7,14,29 \ ,
\end{equation}
where $V_j$ represents the maximum rate achieved by the system at maximum (saturating) substrate
concentrations while $K_j$ is the substrate concentration at which the reaction rate is half the maximum
value. The parameter vector contains all the reaction constants,
\begin{equation}
 \theta = [k_1,\ldots,k_6,k_8,\ldots,k_{13},k_{15},\ldots,k_{28},k_{30},\ldots,k_{47},V_7,K_7,V_{14},K_{14},V_{29},K_{29}]^T \ .
\end{equation}
In this study the values of the reaction constants are the same as in \cite{Kholodenko:99}.

For the initial data and parameters chosen in this study, the time series can be split into two regimes:  (a) a transient regime that approximately corresponds to the
time interval $[0,50]$ and (b) a stationary regime which approximately corresponds to the time interval $[50,\infty)$.

Next, we discuss two sensitivity measures  of the process $ \{X_s\}_{s=0}^t$: the instantaneous RE defined in (\ref{Rel:entr:CTMC}) 
\begin{equation}
\begin{aligned}
f \left( \{X_s\}_{s=0}^t \right) &=  \RELENTR{\IPATHS{t}}{\IPATHSAPP{t}} \\
&= \EXPECT_{\nu_t^\theta} \Big[\sum_{j=1}^M a_j^{\theta}({\bf x})\log \frac{a_j^{\theta}({\bf x})}{a_j^{\theta+\epsilon}({\bf x})}
-\big(a_0^{\theta}({\bf x}) - a_0^{\theta+\epsilon}({\bf x})\big) \Big],
\label{instanteneous:RE}
\end{aligned}
\end{equation}
with $t\in[0,T]$ and the averaged RE, defined as
\begin{equation}
g \left( \{X_s\}_{s=0}^t \right) = \frac{1}{t} \int_0^t \RELENTR{\IPATHS{s}}{\IPATHSAPP{s}} ds, \quad t\in[0,T]  \ .
\label{averaged:RE}
\end{equation}
In Figure (\ref{fig:IRE}) the two  sensitivity measures are presented for $T=100$. As expected, the averaged RE is smoother than the instantaneous RE while some of the qualitative characteristics remain. On the other hand, quantitative characteristics, such as the time that two instantaneous RE are crossed, are not preserved in the averaged RE. The averaged RE in the interval $[0,t]$ should be interpreted as a measure of the information accumulated in the the whole interval while the instantaneous RE is a measure of the information at the time instant $t$. Moreover, the observable can used as part of an upper bound for a different sensitivity measure, see  \cite{AKP:2015} for a detailed discussion.

\begin{figure}[!htb]
\begin{center}
\includegraphics[width=0.49\textwidth]{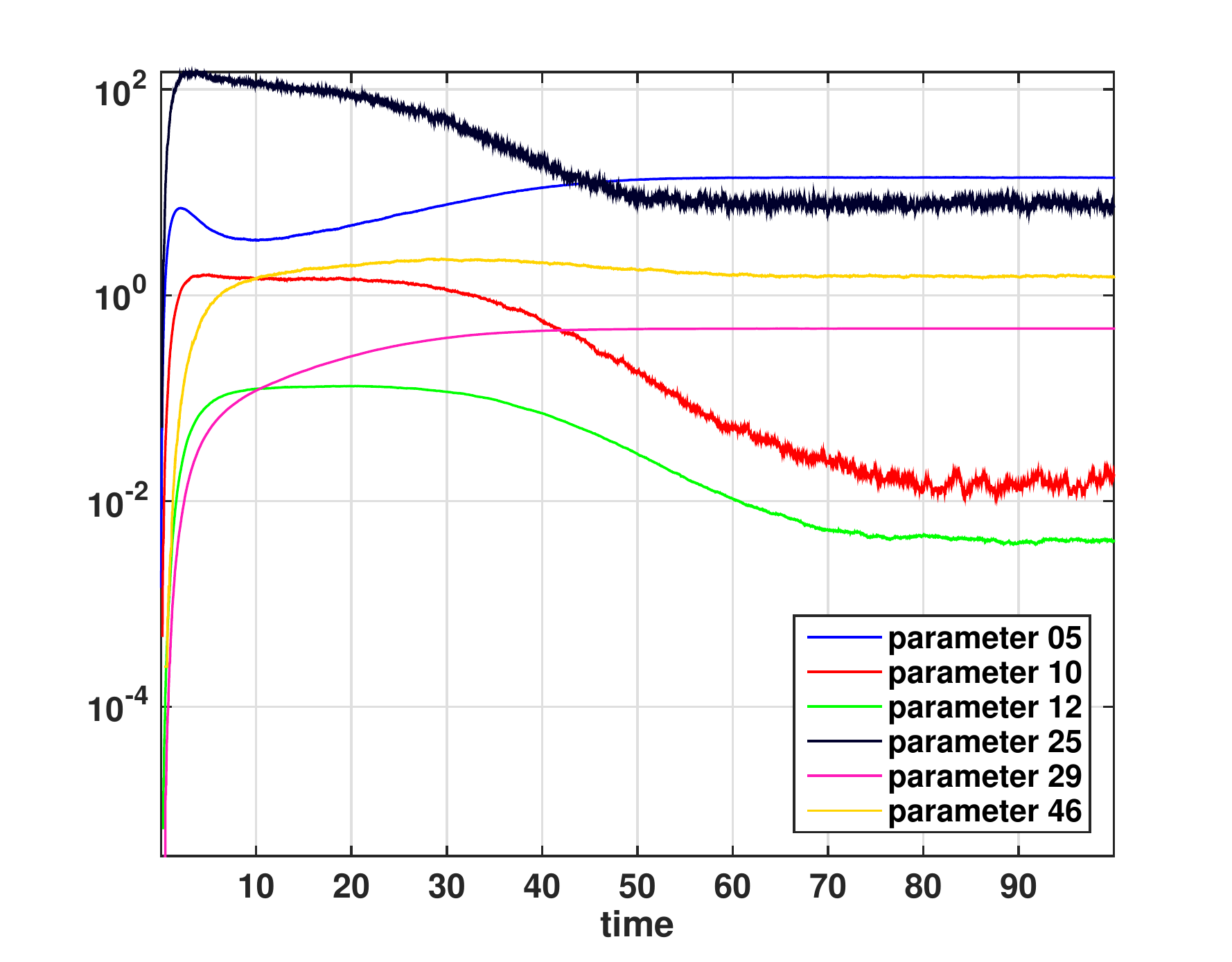}
\includegraphics[width=0.49\textwidth]{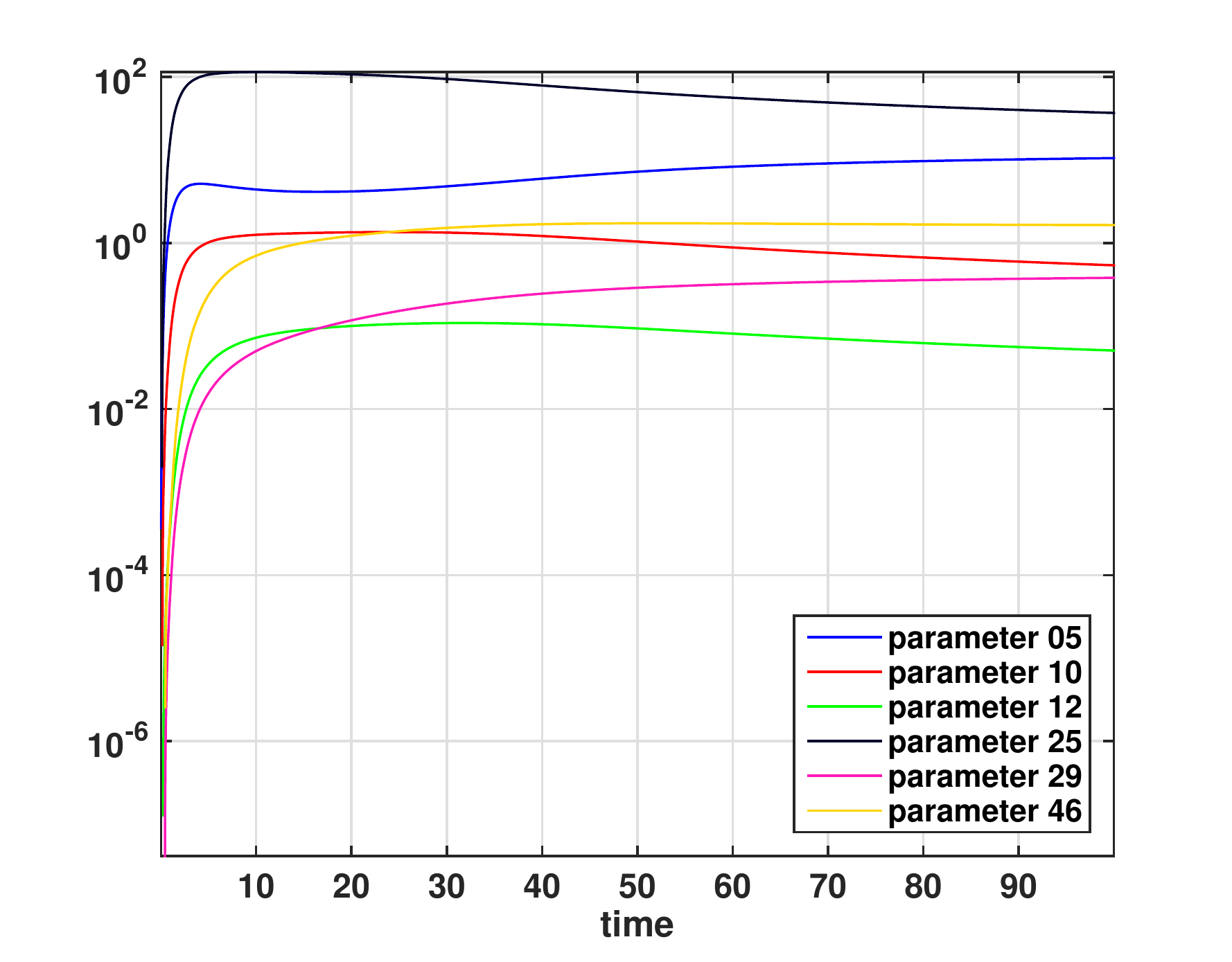}
\caption{The instantaneous RE, defined in (\ref{Rel:entr:CTMC}), for the EGFR model (left) and the averaged RE, defined in (\ref{averaged:RE}) (right).    }
\label{fig:IRE}
\end{center}
\end{figure}

Let us define a different sensitivity measure as the relative difference between the $k$-th species of two systems were the $\ell$-th parameter of the second is perturbed by $\epsilon'$,

\begin{equation}\label{SI}
S_{k,\ell,t} :=   \frac{ X_{k,t}^{\theta} - X_{k,t}^{\theta+\epsilon_\ell} }{ X_{k,t}^{\theta} } \ .
\end{equation}
where $\epsilon_\ell$ is a vector with zeros everywhere and $\epsilon'$ in the $\ell$-th  position. In the following examples the value of $\epsilon'$ for a perturbation in the $\ell$-th parameter is $\epsilon'=0.1\theta_\ell$. By summing over all species we obtain a total sensitivity measure, which is only indexed  by the parameter index and time,
\begin{equation}\label{total:SI}
S_{\ell,t} = \frac{1}{N} \sum_{k=1}^N S_{k,\ell,t} \ .
\end{equation}

 By observing  the first row of Figure \ref{fig:IREvsSENS}, which shows the instantaneous RE (\ref{inst:Rel:entr:CTMC}) for  $\ell=10$ and $\ell=29$, we learn that perturbations in the $10$-th parameter have large sensitivity for small times and as time varies the sensitivity is getting smaller. On the other hand, perturbations in the $29$-th parameter have small influence on the system for small times while for larger times the sensitivity becomes significant. These observations are in good agreement with the sensitivity measure (\ref{SI}) presented in the second and third row of Figure \ref{fig:IREvsSENS}.

\begin{figure}[!htb]
\begin{center}
\includegraphics[width=0.75\textwidth]{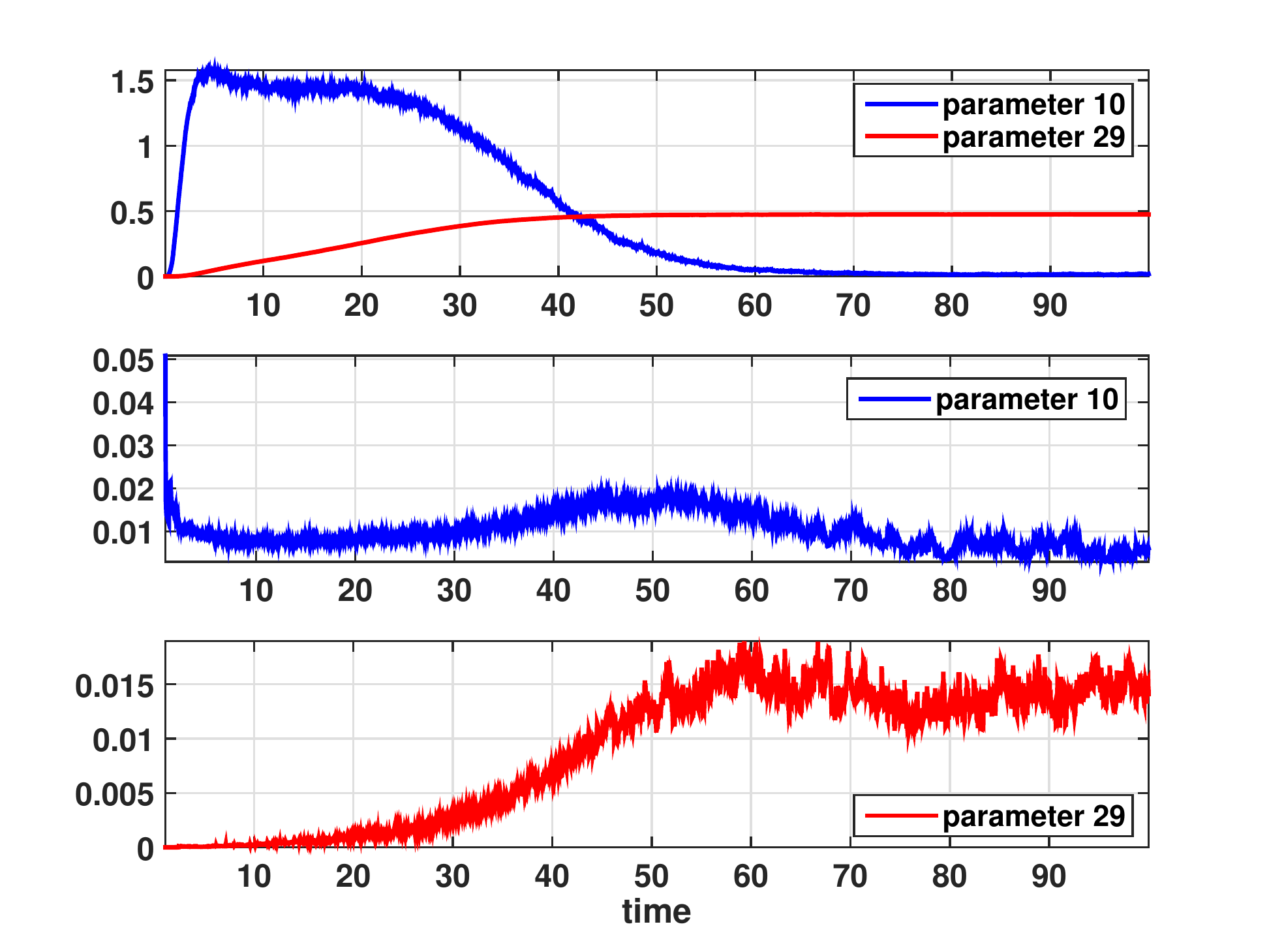}
\caption{Instantaneous RE (\ref{Rel:entr:CTMC}) for the EGFR model  and parameters $10$ and $29$ (first row). The total sensitivity of the system, as defined in (\ref{total:SI}), due to perturbations in the $10$-th parameter (second row) and due to the $29$-th parameter (third row).}
\label{fig:IREvsSENS}
\end{center}
\end{figure}

Moreover there is a crossing in the instantaneous RE which happens around $t=42$. After this time the sensitivity in parameter $\ell=29$ becomes more significant than the sensitivity in parameter $\ell=10$. This is again in agreement with the behavior of the sensitivity measure (\ref{SI}). This observation shows that the transient regime is more sensitive in perturbations in the $10$-th parameters while the equilibrium is more sensitive in perturbations in the $29$-th parameter.

%==============================================================================================================
%  CONCLUSIONS
%==============================================================================================================

\section{Conclusions}

In this paper we presented two pathwise sensitivity measures for the analysis of stochastic systems; the instantaneous relative entropy and its approximation, the instantaneous Fisher information matrix. These sensitivity tools serve as an extension to transient processes of the sensitivity tools presented in  \cite{Pantazis:Kats:13}. Three examples, discrete-time Markov chains, continuous-time Markov Chains and stochastic differential equations, were presented as an application of the new sensitivity measures. In Section 6 we demonstrated, in a biological reaction network, how the proposed pathwise sensitivity measure can be applied to transient, as well as in steady state regimes.

Finally, the pathwise sensitivity method is directly connected to a different sensitivity measure that depends on specific observables. More specifically, if we define the sensitivity index (SI) of the  $\ell$-th observable to the $k$-th parameter as
\begin{equation}
S_{k,\ell} = \frac{\partial }{\partial \theta_k} \EXPECT \left [   f_\ell \left ( \{ X_s  \}_{s=0}^T \right)   \right ] \ ,
\label{SI}
\end{equation}
then IFIM (\ref{inst:FIM:CTMC}) serves as un upper bound of $S_{k,\ell}$ through the inequality,
\begin{equation}
|S_{k,\ell}| \le   \sqrt{\var_{ \PATHS }(f_\ell)} \sqrt{\FISHER \PATHS _{k,k}} \ ,
\label{sens:bound:gen}
\end{equation}
where $F=(f_1,...,f_L)$ is a vector of observable functions.
This inequality follows by rearranging the generalized Cramer-Rao bound for a biased estimator \cite{Casella2002,Kay:93}.
Due to low variance of the estimator of IFIM compared to the variance of a finite difference estimator of $S_{k,\ell}$, the estimation of the right hand side of (\ref{sens:bound:gen}) is faster that that of the left hand side of (\ref{sens:bound:gen}).
In \cite{AKP:2015} the authors use this inequality to efficiently screen out and exclude low sensitivity indices under a pre-specified value and then perform a coupling finite difference algorithm \cite{Anderson:12} to accurately estimate the remaining sensitivity indices $S_{k,\ell}$. In Figure \ref{fig:RE_SI_order} the estimated SIs for the EGFR model discussed in Section \ref{demonstration} are ordered in the parameter direction using only the IFIM (\ref{inst:FIM:CTMC}). Notice that the SIs are then grouped into four distinct regions. For a detailed presentation of this methodology we refer to \cite{AKP:2015}.

\begin{figure}[!htb]
\begin{center}
\includegraphics[width=.49\textwidth]{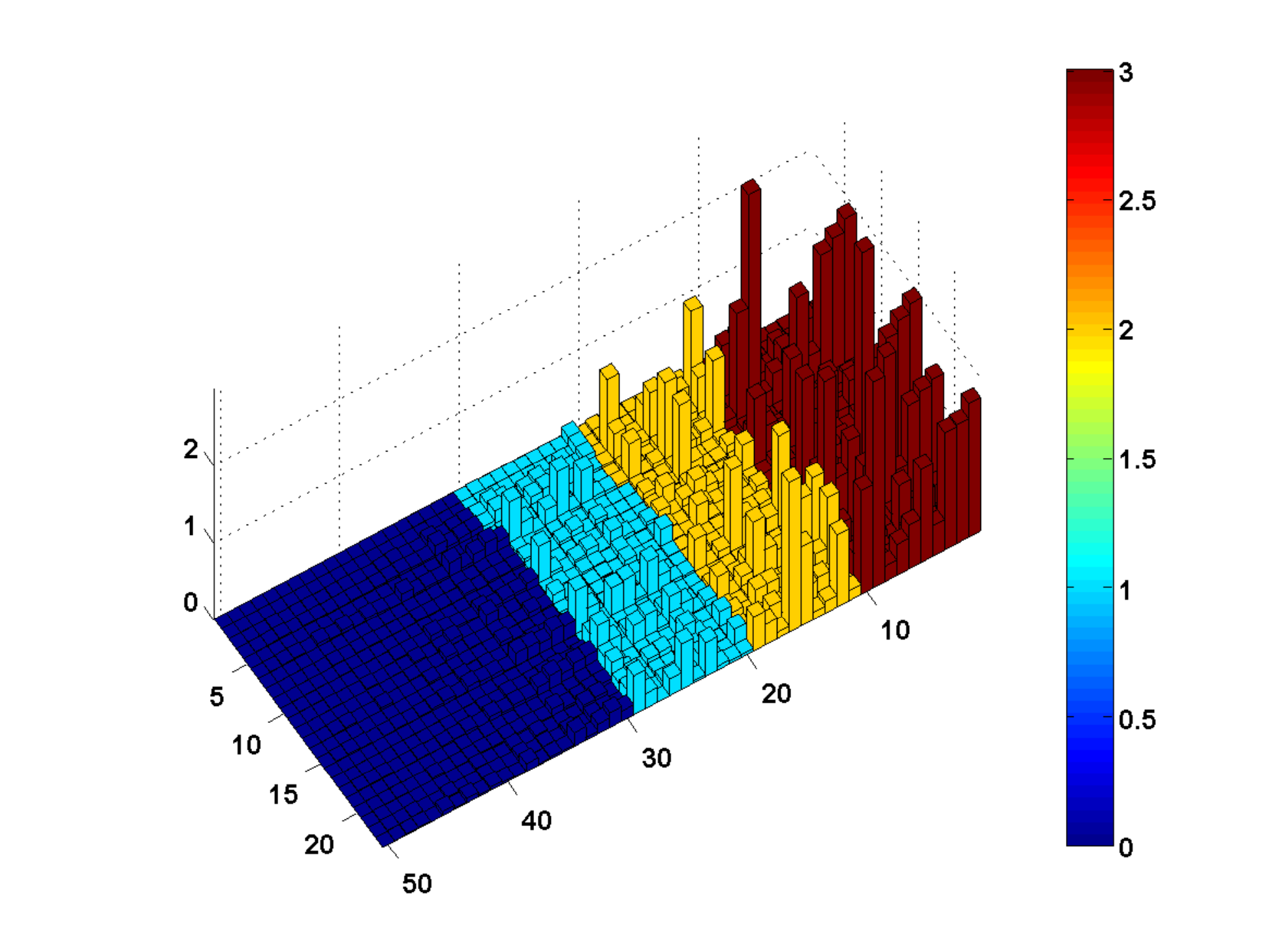}
\includegraphics[width=.49\textwidth]{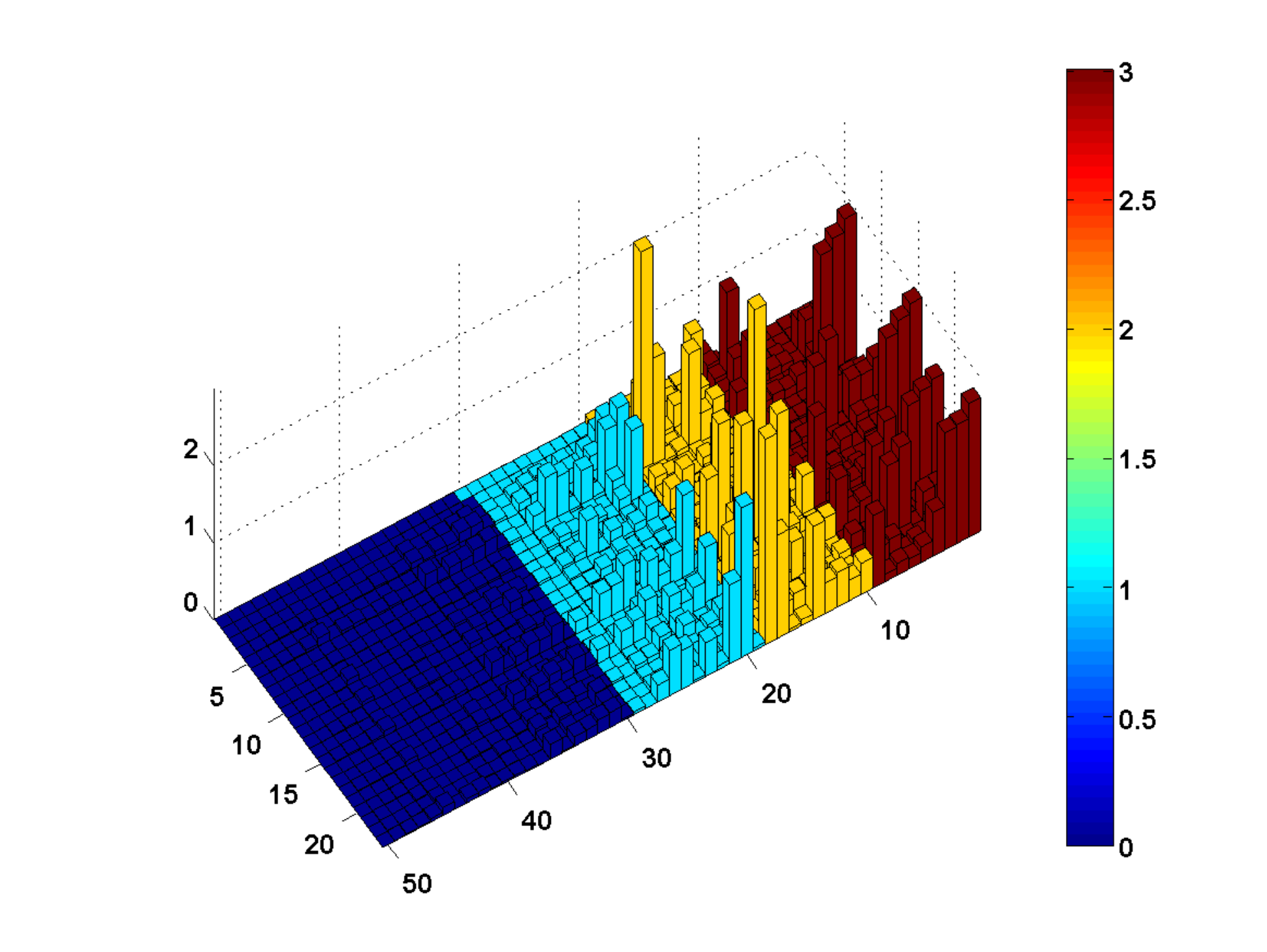}
\caption{Ordering of the sensitivity index (\ref{SI}) in the time interval $[0,50]$ (left) and $[0,100]$ (right) utilizing the averaged IRE (\ref{averaged:RE}). In this case $f_\ell \left ( \{ X_s  \}_{s=0}^T \right) = \frac{1}{T} \int_0^T X_{\ell,s} ds $   , which is the mean concentration of the $\ell$-th species in time interval $[0,T]$.}
\label{fig:RE_SI_order}
\end{center}
\end{figure}

\section*{Acknowledgement}
The work of the authors was supported by the Office of
Advanced Scientific Computing Research, U.S. Department
of Energy, under Contract No. DE-SC0002339 and by the European Union
(European Social Fund) and Greece (National Strategic Reference Framework),
under the THALES Program, grant AMOSICSS.
%The work of MAK    was supported in part by the Office of
%Advanced Scientific Computing Research, U.S. Department
%of Energy under Contracts No. DE-SC0002339.

%\input{referenc}

\bibliographystyle{unsrt}
\bibliography{strategies}

\end{document}